\documentclass{IEEEtran}
\usepackage{amsmath}
\usepackage{graphicx,epstopdf}
\usepackage{cite}
\usepackage{color}
\usepackage{amsfonts}
\usepackage{stfloats}
\usepackage{subfigure}
\usepackage{bm}
\usepackage[utf8]{inputenc}
\usepackage[]{footmisc}
\usepackage{amsthm}
\usepackage{algorithmic}
\usepackage{algorithm}
\usepackage{amsmath,amsthm,amsfonts,amssymb,amscd}
\usepackage{lastpage}
\usepackage{enumerate}
\usepackage{fancyhdr}
\usepackage{mathrsfs}
\usepackage{xcolor}
\usepackage{graphicx}
\usepackage[english]{babel}
\usepackage{hyperref}
\usepackage{bookmark}

\newcommand{\R}{\mathbb{R}}

\newcommand{\C}{\mathcal{C}}
\newcommand{\fot}{\frac{1}{2}}

\newtheorem{theorem}{Theorem}

\newtheorem{assumption}{Assumption}

\newtheorem{lemma}{Lemma}

\newcommand{\pa}{\partial}

\newcommand{\proj}{{\rm Proj}}

\newcommand{\sumjn}{\sum_{j=1}^N}

\def\BibTeX{{\rm B\kern-.05em{\sc i\kern-.025em b}\kern-.08em
    T\kern-.1667em\lower.7ex\hbox{E}\kern-.125emX}}
\begin{document}
\title{Safe Adaptive Multi-Agent Coverage Control}
\author{Yang~Bai$^{1}$, 
        Yujie~Wang$^{2}$,
        Xiaogang~Xiong$^{3}$, 
        and~Mikhail~Svinin$^{4}$

\thanks{This research was supported, in part, by the Japan Science and Technology Agency, the JST Strategic International Collaborative Research Program, Project No. 18065977, and by the Russian Foundation for Basic Research, Project No. 19-58-70002.}

\thanks{$^{1}$Y. Bai is with the Graduated School of Information Science and Technology,
Osaka University 1-5 Yamada, Suita, Osaka 565-0871, Japan (email: y-bai@ist.osaka-u.ac.jp).}
        
\thanks{$^{2}$Y. Wang is with the Department of Mechanical Engineering, University of Wisconsin-Madison, Madison, WI, USA (e-mail: ywang2835@wisc.edu).}

\thanks{$^{3}$X. Xiong is with the school of mechanical engineering and automation, Harbin Institute of Technology, Shenzhen, China\protect\\
(e-mail: xiongxg@hit.edu.cn).}

\thanks{$^{4}$M. Svinin are with the Information Science and Engineering Department,
    Ritsumeikan University,
    1-1-1 Noji-higashi, Kusatsu,
    Shiga 525-8577, Japan
        (e-mail: svinin@fc.ritsumei.ac.jp).}
}

\maketitle

\begin{abstract}
  This paper presents a safe adaptive coverage controller for multi-agent systems with actuator faults and time-varying uncertainties. The centroidal Voronoi tessellation (CVT) is applied to generate an optimal configuration of multi-agent systems for covering an area of interest. As conventional CVT based controller cannot prevent collisions between agents with non-zero size, a control barrier function (CBF) based controller is developed to ensure collision avoidance with a function approximation technique (FAT) based design to deal with system uncertainties. The proposed controller is verified under simulations.
\end{abstract}


\section{Introduction}
The research on multi-agent coverage problems has received a considerable attention in the recent decades\cite{Cortes_04,Lee_15,Miah_17b,Chevet_19,Alessia_20,Li_20}.
The controllers proposed in\cite{Cortes_04,Lee_15,Miah_17b} can react to dynamic environments with time-varying density functions,
and the research conducted in\cite{Chevet_19,Alessia_20,Li_20} dealt with systems with uncertainties.

These algorithms are based on CVT, representing agents by points, that stay in their corresponding cells to avoid collisions with each other. When the algorithms are implemented in real robots however, collisions may still occur unless the radius of the robots is zero, which is impractical. To tackle this problem, several works based on the buffered Voronoi cell (BVC) have been proposed in\cite{Zhou_17,Wang_18_BVC,Pierson_20}. The BVC retracts the edge of each Voronoi cell by a safety radius, such that collisions are avoided if agents stay in the cells\cite{Zhou_17}. Nevertheless, this redesign of the Voronoi tessellation included path planning and tracking in each step. The collision avoidance capability of this process has not been tested in the presence of system uncertainties, which may drive agents out of the BVC and cause collisions.

Different from our previous work\cite{Bai_22a}, in this paper, we present a solution to the collision avoidance problem in the coverage control based on the CBF\cite{ames2016control}, which does not require a redesign of the Voronoi tessellation. Instead, safety range for each agent is constructed such that the collision avoidance between agents can be guaranteed. Moreover, we consider actuator faults and time-varying disturbances for the agents in the controller design, which could seriously affect the system performance in a safety critical problem.

Conventional CBF approaches require accurate model information and are commonly inapplicable to control systems with uncertainties. Robust CBF methods\cite{garg2021robust,nguyen2021robust,verginis2021safety,buch2021robust} were developed to attenuate the influence of uncertainties whereas they rely on assumptions for boundness or Lipschitzness of the uncertainty.
Apart from robust approaches, several adaptive CBF methods have been proposed for systems with parametric uncertainties \cite{taylor2020adaptive,lopez2020robust,isaly2021adaptive,cohen2022high}. The control scheme in \cite{taylor2020adaptive} guaranteed safety using an online parameter adaption, and the strategy in \cite{lopez2020robust} allowed systems to operate in larger safe sets, reducing conservatism. The concurrent learning technique based controller in \cite{isaly2021adaptive} ensured the exponential convergence of the parameter estimation, and a high-order robust adaptive CBF has been developed in \cite{cohen2022high} for systems with high-relative-degree safety constraints.
To the best of our knowledge however, none of the CBF based techniques above is applicable to our system, which includes uncertainties in both the input matrix (actuator faults) and the drift term (time-varying disturbances).

Note that although FAT can transform time-varying uncertainties to parametric ones, the control of our system is still challenging because of the inclusion of actuator faults. It may cause a controllability loss issue\cite{bechlioulis2008robust} due to that the transient response of update laws, for uncertain parameters in the input matrix, is commonly uncontrollable. Techniques that addressed this issue, such as the Nussbaum gain approach \cite{liu2017barrier}, have been proposed for stabilization problems, whereas systematic methods for CBF related problems have not been developed yet. In this paper, a novel adaptive CBF based approach is proposed that addresses the controllability loss issue with a special design of the update laws, and deals with the time-varying disturbances based on the FAT.

The rest of the paper is organized as follows. In Section~\ref{sec:review}, some preliminaries about CVT, FAT and CBF are introduced. In Section~\ref{sec:control},
a safe adaptive coverage problem is stated for multi-agent systems, and a corresponding control algorithm is developed based on the CBF technique.
The proposed control algorithm is verified under simulations in Section~\ref{sec:case},
and finally, conclusions are drawn in Section~\ref{sec:concl}.


\section{Preliminaries}
\label{sec:review}
In this section, the concepts about FAT, CBF, and CVT are reviewed, which are the main tools for our controller design.

\subsection{CBF}
\label{sec:cbf}
Consider the following control affine system
\begin{IEEEeqnarray}{rCl}
\dot{\bm{x}} &= & \bm{f}(\bm{x}) + \bm{g}(\bm{x}) \bm{u},\label{eqnsys}
\end{IEEEeqnarray}
where $\bm{x}\in\mathbb{R}^n$ is the state, $\bm{u}\in U\subset\mathbb{R}^m$ is the control input, $\bm{f}: \mathbb{R}^n\to\mathbb{R}^n$ and $\bm{g}:\mathbb{R}^n\to\mathbb{R}^{n\times m}$ are locally Lipchitz continuous functions.
A set $\mathcal{S}$ is called forward controlled invariant with respect to system \eqref{eqnsys} if for every $\bm{x}_0 \in \mathcal{S}$, there exists a control signal $\bm{u}(t)$ such that $\bm{x}(t;t_0,\bm{x}_0) \in \mathcal{S}$ for all $t\geq t_0$, where $\bm{x}(t;t_0,\bm{x}_0)$ denotes the solution of \eqref{eqnsys} at time $t$ with initial condition $\bm{x}_0\in\mathbb{R}^n$ at time $t_0$.

Consider control system \eqref{eqnsys} and a set $\mathcal{C} \subset \R^n$ defined by
\begin{equation}\label{setc}
    \mathcal{C} = \{ \bm{x} \in \R^n : h(\bm{x}) \geq 0\},
\end{equation}
for a continuously differentiable function $h: \R^n \to \R$ that has a relative degree one. The function $h$ is called a (zeroing) CBF
if  there exists a constant $\gamma>0$ such that
\begin{align}\label{ineqZCBF}
& \sup_{\bm{u} \in U}  \left[ \mathcal{L}_f h(\bm{x}) + \mathcal{L}_g h(\bm{x}) \bm{u} + \gamma h(\bm{x})\right] \geq 0, 
\end{align}
where $\mathcal{L}_fh(\bm{x})=\frac{\partial h}{\partial \bm{x}}\bm{f}(\bm{x})$, $\mathcal{L}_gh(\bm{x})=\frac{\partial h}{\partial \bm{x}}\bm{g}(\bm{x})$ are the Lie derivatives \cite{Xu2015ADHS}.
Given a CBF $h$, the set of all control values that satisfy \eqref{ineqZCBF} for all $\bm{x}\in\R^n$ is defined as
$K_{bf}(\bm{x}) =  \{ \bm{u} \in U : \mathcal{L}_f h(\bm{x}) + \mathcal{L}_g h(\bm{x}) \bm{u} + \gamma h(\bm{x}) \geq 0\}.$
It was proven in \cite{Xu2015ADHS} that any Lipschitz continuous controller $\bm{u}(\bm{x}) \in K_{bf}(\bm{x})$ for every $\bm{x}\in\R^n$ will guarantee the forward invariance of $\mathcal{C}$. The provably safe control law is obtained by solving an online quadratic program (QP) problem that includes the control barrier condition as its constraint.


\subsection{FAT}\label{sec:FAT}
The FAT is an effective tool to deal with control systems with time-varying nonlinear uncertainties\cite{huang2001sliding}.
For instance, if $\boldsymbol{d}(t)$ is an unknown time-varying function in a control system, one can utilize weighted basis functions to represent $\boldsymbol{d}(t)$, at each time instant, as \cite{huang2001sliding
,Zirkohi_18,Schwager_09,Schwager_17
,Bai_20a,Wang_21,Bai_22a}
\begin{IEEEeqnarray}{rCl}\label{d_approx}
  \boldsymbol{d}(t)=\sum^{\infty}_{j=1}\boldsymbol{d}_{j}\psi_{j}(t),
\end{IEEEeqnarray}
where $\boldsymbol{d}_{j}$ denotes an unknown constant vector (weight) and $\psi_{j}(t)$ is the basis function to be selected. An update law $\hat{\boldsymbol{d}}_{j}(t)$ is designed to approximate $\boldsymbol{d}_{j}$ so as to reject the effect of $\boldsymbol{d}(t)$ to the control system.

Several candidates for the basis function $\psi_{j}(t)$ in \eqref{d_approx} can be chosen to approximate the nonlinear functions, and in this paper, we select the basis function $\psi_{j}(t)$ as the Fourier series\cite{huang2001sliding}
\begin{equation}
\psi_{j}(t)=\begin{cases}
1, & j=1, \\
\cos \omega_{l} t, & j=2l, \\
\sin \omega_{l} t, & j=2l+1,
\end{cases}
\end{equation}
where $\omega_{l}=2\pi l/T$ are the frequencies, $l \in \{1,2,...,(N-1)/2\}$ ($N$ is an odd number), and $T$ is the total time interval.

\subsection{CVT}
Let $\mathcal{\boldsymbol{D}}\subset \mathbb{R}^2 $ be a 2-D domain to be covered by $n$ agents.
Let $\boldsymbol{p}_{i}\in \mathcal{\boldsymbol{D}}, i = 1,...,n$ be the position of the $i$th agent and $\boldsymbol{p}=\{\boldsymbol{p}_{i}\}$ where $\{.\}$ denotes a collection of functions. The coverage problem is concerned with placing the agents in $\mathcal{\boldsymbol{D}}$, dividing $\mathcal{\boldsymbol{D}}$ into regions of dominance of each agent $i$. A Voronoi tessellation is thus formed as
\begin{equation}\label{voronoi}
V_{i}\left ( \boldsymbol{p} \right )= \left \{ \boldsymbol{q}\in \mathcal{\boldsymbol{D}}\mid \left \| \boldsymbol{q}-\boldsymbol{p}_{i}  \right \|\leq \left \| \boldsymbol{q}-\boldsymbol{p}_{j} \right \| ,i\neq j\right \},
\end{equation}
where $\|.\|$ represents the $l_{2}$-norm for a vector,
such that agent $i$ is in charge of each subregion $V_{i}$.

%

To measure the performance of the coverage for the multi-agent system (how well a given point $\boldsymbol{q}\in \mathcal{\boldsymbol{D}}$ is covered by agent $i$ at position $\boldsymbol{p}_{i}\in \mathcal{\boldsymbol{D}}$),
one can define the locational cost
\begin{equation}\label{H0}
\mathcal{H}\left (\boldsymbol{p},t \right)= \sum_{i=1}^{n}\int_{V_{i}}\left \| \boldsymbol{q}-\boldsymbol{p}_{i} \right \|^{2}\phi \left ( \boldsymbol{q} \right )\textrm{d}\boldsymbol{q},
\end{equation}
where $\phi(\boldsymbol{q})$ denotes the associated density function, which is assumed to be positive and bounded.
It captures the relative importance of a point $\boldsymbol{q}\in \mathcal{\boldsymbol{D}}$\cite{Lee_15}.

From \eqref{H0}, the time derivative of $\mathcal{H}$ can be expressed by
\begin{equation}\label{dHdt}
  \dot{\mathcal{H}}=\sum_{i=1}^{n}\frac{\partial \mathcal{H}}{\partial \boldsymbol{p}_{i}}\dot{\boldsymbol{p}}_{i}.
\end{equation}
Note that as the density $\phi$ in time-invariant, $\frac{\partial \mathcal{H}}{\partial t}=0$. In\cite{Lee_15}, it was shown that
\begin{equation} \label{E5}
\frac{\partial \mathcal{H}}{\partial \boldsymbol{p}_{i}}= 2m_{i} ( \boldsymbol{p}_{i}-\boldsymbol{c}_{i}  )^{\top},
\end{equation}
where the mass $m_{i}$ and the center of mass $\boldsymbol{c}_{i}$ of the $i$-th Voronoi cell $V_{i}$ are defined as
$
m_{i}= \int_{V_{i}}\phi \left ( \boldsymbol{q} \right )\textrm{d}\boldsymbol{q},\
\boldsymbol{c}_{i}= \frac{\int_{V_{i}}\phi \left ( \boldsymbol{q} \right )\boldsymbol{q} \textrm{d}\boldsymbol{q}}{m_{i}}
$.
Note that $m_{i}>0$ because the density function $\phi(\boldsymbol{q})$ is strictly positive\cite{Schwager_09}.

At a given time $t$, an optimal coverage for the domain $\mathcal{\boldsymbol{D}}$ requires a configuration of agents $\boldsymbol{p}$ to minimize $\mathcal{H}$.
From \eqref{E5}, one can see that a critical point is
\begin{equation} \label{E6}
\boldsymbol{p}_{i}\left ( t \right )=\boldsymbol{c}_{i}\left ( \boldsymbol{p},t \right ),\: \:  i= 1,...,n.
\end{equation}
When (\ref{E6}) is satisfied, an agents' network is said to be in a locally optimal coverage configuration\cite{Schwager_09}. The corresponding $\boldsymbol{p}$ defines the CVT.

\section{Controller design}\label{sec:control}
In this section, a coverage problem is first defined. Then, a safe adaptive controller based on the FAT and CBF techniques is proposed for a multi-agent system which guarantees the collision avoidance in the presence of system uncertainties.

\subsection{Control problem formalization}
The following single integrator model is adopted for the agents in the control strategy design
\begin{equation}\label{model_f}
 \dot{\boldsymbol{p}}_{i} = \overline{\boldsymbol{u}}_{i}+\boldsymbol{d}_{i}(t),
  \
  i \in \{1,2,...,n\},
\end{equation}
where $n$ is the number of agents, $\boldsymbol{p}_{i}\in\mathbb{R}^{2}$ represents the state of an agent, $\overline{\boldsymbol{u}}_{i}\in\mathbb{R}^{2}$ denotes the input signal, and $\boldsymbol{d}_{i}\in\mathbb{R}^{2}$ is a time-varying uncertainty which is assumed to be continuous and bounded.

Note that the expected system performance and stability can be severely compromised when the actuators do not operate as desired (actuator fault)\cite{Ye_06,Liu_17}. In order to improve the performance and reliability of systems, control methods need to be designed to compensate for the effects of actuator fault, which leads to a fault-tolerant control problem. To study this problem, the following actuator fault model is introduced (when the partial loss of actuator effectiveness is considered)\cite{Ye_06,Liu_17}
\begin{equation}\label{tolerant}
  \overline{\boldsymbol{u}}_{i}=\theta_{i}\boldsymbol{u}_{i},
\end{equation}
and \eqref{model_f} can be rearranged as
\begin{equation}\label{de}
    \dot{\boldsymbol{p}}_{i} =\theta_{i}\boldsymbol{u}_{i}+\boldsymbol{d}_{i}.
\end{equation}
\begin{assumption}\label{assumption1}
The actuator fault coefficient $\theta_{i}$ satisfies $\theta_{i} \in [\alpha_{i}, 1]$, where $\alpha_{i}$ is a small positive constant.
The motivation of imposing this assumption is to guarantee the controllability of system \eqref{de}. If $\theta_{i}=0$, the system is uncontrollable, and the controller design procedure becomes invalid.
\end{assumption}

In the conventional design of CVT based coverage control algorithms, agents (robots) are represented by points that stay in their corresponding cells to avoid collisions with each other. However, when the algorithms are implemented in real robots, collisions may still occur unless the radius of the robots is zero. To deal with this safety issue, in this paper, we design a safe adaptive controller for system \eqref{de} based on the CBF technique. According to Section \ref{sec:cbf}, one defines a safe set $\C_{i}$ for agent $i$ as
\begin{equation}\label{safety_set}
  \C_{i}=\{\bm{p}_{i}\in\R^n, h_{i}(\bm{p}_{i})\geq 0\},
\end{equation}
where $h_{i}$ is a Lipschitz continuous function with a relative degree 1. To avoid collisions, one can design $h_{i}$ as
\begin{equation}\label{}
  h_{i}=\|\bm{p}_{i}-\bm{z}_{i}\|^{2}-(2r_{\textrm{safe}})^{2},
\end{equation}
where $\bm{z}_{i}$ denotes the position of the nearest detected neighbor in the sensing range of agent $i$, and the velocity of the neighbor $\dot{\bm{z}}_{i}$ is measurable by agent $i$, and $r_{\textrm{safe}}$ is the safety radius which the size of agent $i$ can fit into.

Assuming that there exists a nominal control input $\hat{\bm{u}}_{i}$ for $n$ agents to achieve a CVT within $\mathcal{\boldsymbol{D}}$ ($\lim_{t \to \infty}\boldsymbol{p}_{i}-\boldsymbol{c}_{i} = \bm{0}$) regardless to the safety concerns, the control problem can be stated as follows. Construct the actual input $\boldsymbol{u}_{i}$ such that the corresponding $\bm{p}_{i}(t)$ resulted from system \eqref{de} will stay inside the safety set $\C_{i}$ defined as \eqref{safety_set}, i.e., $h(\bm{p}_{i}(t))\geq 0$ for $\forall t>0$, in the presence of a parametric uncertainty $\theta_{i}$ and a time-varying uncertainty $\boldsymbol{d}_{i}$.

To the best of our knowledge, the CBF based safe adaptive controller design for systems with both the time-varying uncertainties and the actuator faults, has not been proposed yet in the literature, which constitutes the novelty of this work.

\subsection{A safe adaptive coverage controller design}
To compensate for the effects of time-varying uncertainty $\boldsymbol{d}_{i}$ in system \eqref{de}, following the FAT approach (see Section \ref{sec:FAT}), the approximation of \eqref{de} can be represented as
\begin{equation}\label{syseqn}
  \dot{\boldsymbol{p}}_{i} = \theta_{i}\boldsymbol{u}_{i}+\sum^{N}_{j=1}\boldsymbol{d}_{ij}\psi_{j}(t)+\boldsymbol{\epsilon}_{i},
\end{equation}
where $\boldsymbol{d}_{ij}$ denotes an unknown constant vector (weight), $\psi_{j}(t)$ is the basis function to be selected, and $\boldsymbol{\epsilon}_{i}$ describes the deviation between the uncertainty $\boldsymbol{d}_{i}$ and the weighted basis functions.
\begin{assumption}
The unknown parameter $\bm{d}_{ij}$, error $\boldsymbol{\epsilon}_{i}$, and velocity $\dot{\bm{z}}_{i}$ are bounded such that $\|\bm{d}_{ij}\|\leq \bar{d}_{ij}$, $\|\bm{\epsilon}_{i}\| \leq E$, $\|\dot{\bm{z}}_{i}\| \leq V_{z}$, where $\bar{d}_{ij}$, $E$, and $V_{z}$ are positive constants.
\end{assumption}
Regarding the actuator faults, setting $\theta_{i}=\theta_{i}^*+\frac{1+\alpha_{i}}{2}$, we can see that $\theta_{i}^*\in[\frac{\alpha_{i}-1}{2}, \frac{1-\alpha_{i}}{2}]$.
By selecting $\bar{\theta}_{i}=\frac{1-\alpha_{i}}{2}$, we have $\theta_{i}^*\in[-\bar{\theta}_{i}, \bar{\theta}_{i}]$. Then, system \eqref{syseqn} becomes
\begin{equation}
        \dot{\bm{p}}_{i} =\theta_{i}^* \bm{u}_{i}+\frac{1+\alpha_{i}}{2} \bm{u}_{i}+\sum^{N}_{j=1}\boldsymbol{d}_{ij}\psi_{j}(t)+\boldsymbol{\epsilon}_{i}.
        \label{syseqn1}
\end{equation}
To compensate for the effects of the unknown parameters $\theta_{i}^*$ and $\bm{d}_{ij}$ to the control system \eqref{syseqn1}, corresponding update laws $\hat{\theta}_{i}$ and $\hat{\bm{d}}_{ij}$ need to be designed for approximating these parameters. A projection operator is adopted in the consequent design of update laws, which is defined as\cite{lavretsky2011projection}
\begin{IEEEeqnarray}{rCl}
&&\proj(\bm{x},\bm{y},l) \nonumber\\
    &&=\begin{cases}
     \bm{y}-l(\bm{x})\frac{\nabla l(\bm{x})\nabla l(\bm{x})^\top}{\|\nabla l(\bm{x})\|^2} \bm{y}, & \text{if} \ l(\bm{x})>0\ \!\wedge\ \! \bm{y}^\top \nabla l(\bm{x})>0, \\
     \bm{y}, & \text{otherwise}.
    \end{cases} \nonumber
\end{IEEEeqnarray}
Here $\bm{x}$, $\bm{y}$ are arbitrary vectors and $l(\bm{x})$ is convex function defined as
\begin{equation}
    l(\bm{x})=\frac{\bm{x}^\top \bm{x}-\bar{x}^2}{2\eta \bar{x}+\eta^2},
\end{equation}
where $\bar{x}$ and $\eta$ are constants.
\begin{lemma}\label{lemma1}
Given $\bm{x}^*$ as the nominal value for $\bm{x}$ which may vary from $\bm{x}$,
if $l(\bm{x}^*)\leq 0$, one can see that\cite{lavretsky2011projection}
\begin{equation}
    (\bm{x}-\bm{x}^*)^{\top}(\proj(\bm{x},\bm{y},l)-\bm{y})\leq 0.\label{proj}
\end{equation}
\end{lemma}
\begin{lemma}\label{lemma2}
If $\bm{x}(t)$ is governed by
\begin{equation}
    \dot{\bm{x}}=\proj(\bm{x},\bm{y},l)\label{projector},
\end{equation}
and $l(\bm{x}(0))\leq 1$, one can get
$l(\bm{x}(t))\leq 1$ for $\forall t\geq 0$, which implies $\|\bm{x}(t)\|\leq \bar{x} + \eta$\cite{lavretsky2011projection}.
\end{lemma}
\begin{theorem}\label{theorem1}
By constructing the update laws $\hat{\theta}_{i}$ and $\hat{\bm{d}}_{ij}$ for the parameter estimation as
\begin{IEEEeqnarray}{rCl}
    \dot{\hat{\theta}}_{i}
    &=&\proj \Bigg(\hat{\theta}_{i},
    -\frac{1}{2K_{i}} \bigg(\frac{\pa h_{i}}{\pa \bm{p}_{i}}\bigg)^{\top}\bm{u}_{i}-\frac{\mu}{2} \hat{\theta}_{i}, l_{\theta i}\Bigg), \label{adaptivelaw1} \\
    \dot{\hat{\bm{d}}}_{ij}&=&\proj\Bigg(\hat{\bm{d}}_{ij},-\frac{1}{2Q_{ij}} \bigg(\frac{\pa h_{i}}{\pa \bm{p}_{i}}\bigg) \psi_{j}-\frac{\mu}{2} \hat{\bm{d}}_{ij}, l_{di}\Bigg), \label{adaptivelaw2}
\end{IEEEeqnarray}
where
\begin{IEEEeqnarray}{rCl}\label{l}
    l_{\theta i}(\hat{\theta}_{i})&=&\frac{\hat{\theta}_{i}^{2}-\bar{\theta}_{i}^2}{\bar{\theta}_{i}\alpha_{i} +(\alpha_{i}/2)^2},\
    l_{di}(\hat{\bm{d}}_{ij}) = \frac{\hat{\bm{d}}_{ij}^{\top}\hat{\bm{d}}_{ij}-\bar{d}_{ij}^2}{2\nu_{i} \bar{d}_{ij}+\nu_{i}^2},
\end{IEEEeqnarray}
$\alpha_{i}$ and $\nu_i$ are small constants, $\|\hat{\theta}_{i}(0)\|\leq \frac{1}{2}$, and
\begin{equation}
    K_{i}\leq \frac{h_{i}(\bm{p}_{i}(0))}{2(\|\hat{\theta}_i(0)\|+\bar{\theta}_{i})^2}, \
    Q_{ij}\leq \frac{h_{i}(\bm{p}_{i}(0))}{2N(\|\hat{\bm{d}}_{ij}(0)\|+\bar{d}_{ij})^2},\label{adaptiveparameters}
\end{equation}
any Lipschitz continuous controller $\bm{u}_{i}(\bm{p}_{i})\in K_{bf}(\bm{p}_{i}, \hat{\theta}_{i},\hat{\bm{d}}_{ij})$ where
\begin{IEEEeqnarray}{rCl}
&&K_{bf}(\bm{p}_{i}, \hat{\theta}_{i},\hat{\bm{d}}_{ij})\nonumber\\
&&\triangleq\bigg\{\bm{u}_{i}\in\mathbb{R}^m \, | \, \bigg(\frac{\pa h_{i}}{\pa \bm{p}_{i}}\bigg)^{\!\!\top} \!\sum_{i=1}^N \hat{\bm{d}}_{ij} \psi_{j}
+(\hat{\theta}_{i}+\frac{1+\alpha_{i}}{2})\bigg(\frac{\pa h_{i}}{\pa \bm{p}_{i}}\bigg)^{\!\!\top} \!\bm{u}_{i} \nonumber\\
&&\quad\quad
-\zeta_{i}+\frac{\mu}{2}\bigg(h_{i}-K_{i}\bar{\theta}_{i}^2-\sumjn Q_{ij} \bar{d}_{ij}^2\bigg)
\geq 0\bigg\}, \label{kcbf}
\end{IEEEeqnarray}
with $\zeta_{i}=\Big\|\frac{\pa h_{i}}{\pa \bm{p}_{i}}\Big\|E
+\Big\|\frac{\pa h_{i}}{\pa \bm{z}_{i}}\Big\|V_{z}$,
will guarantee the safety of $\C_{i}$ in regard to system \eqref{syseqn}.
\end{theorem}

\begin{proof}
Define $\bar{h}_{i}$ as
\begin{equation}\label{hbar}
    \bar{h}_{i}=h_{i}-K_{i}\tilde{\theta}_{i}^2-\sum_{j=1}^N Q_{ij}\tilde{\bm{d}}_{ij}^\top \tilde{\bm{d}}_{ij},
\end{equation}
where $\tilde{\theta}_{i}=\theta_{i}^{*}-\hat{\theta}_{i}$ and $\tilde{\bm{d}}_{ij}=\bm{d}_{ij}-\hat{\bm{d}}_{ij}$. To prove Theorem \ref{theorem1}, one needs to show that $\bar{h}_{i}(t)\geq0$ for $\forall t$, such that $h_{i}(t)\geq 0$ for $\forall t$ as required by \eqref{safety_set}. This property holds if $\dot{\bar{h}}_{i}$ can be expressed in form of (or larger than) $-\lambda \bar{h}_{i}$ where $\lambda>0$ with $\bar{h}_{i}(0)\geq 0$.

A reconstruction of $\dot{\bar{h}}_{i}$ to the form of $-\lambda \bar{h}_{i}$ is demonstrated as follows.
It can be seen that $\dot{\bar{h}}_{i}$ is calculated as
\begin{IEEEeqnarray}{rCl}
\dot{\bar{h}}_{i} &=& \bigg(\frac{\pa h_{i}}{\pa \bm{p}_{i}}\bigg)^\top \dot{\bm{p}}_{i}
+\bigg(\frac{\pa h_{i}}{\pa \bm{z}_{i}}\bigg)^\top \dot{\bm{z}}_{i}
- 2 K \tilde{\theta}_{i}\dot{\tilde{\theta}}_{i}- 2\sumjn Q_{ij} \tilde{\bm{d}}_{ij}^\top \dot{\tilde{\bm{d}}}_{ij}, \nonumber\\
&=& \bigg(\frac{\pa h_{i}}{\pa \bm{p}_{i}}\bigg)^\top\bigg(\theta_{i}^* \bm{u}_{i}+\frac{1+\alpha_{i}}{2} \bm{u}_{i}+\sum^{N}_{j=1}\boldsymbol{d}_{ij}\psi_{j}(t)+\boldsymbol{\epsilon}_{i}\bigg) \nonumber\\
&&+\bigg(\frac{\pa h_{i}}{\pa \bm{z}_{i}}\bigg)^\top \dot{\bm{z}}_{i}
+ 2 K \tilde{\theta}_{i}\dot{\hat{\theta}}_{i}+ 2\sumjn Q_{ij} \tilde{\bm{d}}_{ij}^\top \dot{\hat{\bm{d}}}_{ij}\nonumber\\
&\geq&
\bigg(\frac{\pa h_{i}}{\pa \bm{p}_{i}}\bigg)^\top\bigg(\theta_{i}^* \bm{u}_{i}+\frac{1+\alpha_{i}}{2} \bm{u}_{i}+\sum^{N}_{j=1}\boldsymbol{d}_{ij}\psi_{j}(t)\bigg)
-\zeta_{i} \nonumber\\
&&
+ 2 K \tilde{\theta}_{i}\dot{\hat{\theta}}_{i}+ 2\sumjn Q_{ij} \tilde{\bm{d}}_{ij}^\top \dot{\hat{\bm{d}}}_{ij}.\label{doth1}
\end{IEEEeqnarray}
As update laws $\dot{\hat{\theta}}_{i}$ and $\dot{\hat{\bm{d}}}_{ij}$ in \eqref{doth1} are defined as \eqref{adaptivelaw1} and \eqref{adaptivelaw2},
from Lemma \ref{lemma1}, one can see
\begin{IEEEeqnarray}{rCl}
\tilde{\theta}_{i}\dot{\hat{\theta}}_{i} &=& (\theta_{i}^{*}-\hat{\theta}_{i})
    \proj \Bigg(\hat{\theta}_{i},
    -\frac{1}{2K_{i}} \bigg(\frac{\pa h_{i}}{\pa \bm{p}_{i}}\bigg)^{\top}\bm{u}_{i}-\frac{\mu}{2} \hat{\theta}_{i}, l_{\theta i}\Bigg) \nonumber\\
    &\geq& -(\theta_{i}^{*}-\hat{\theta}_{i})\Bigg(\frac{1}{2K_{i}} \bigg(\frac{\pa h_{i}}{\pa \bm{p}_{i}}\bigg)^{\top} \bm{u}_{i}+\frac{\mu}{2} \hat{\theta}_{i}\Bigg), \label{projadaptive1} \\
\tilde{\bm{d}}_{ij}^\top \dot{\hat{\bm{d}}}_{ij} &=& (\bm{d}_{ij}-\hat{\bm{d}}_{ij})^{\top} \nonumber\\
    && \proj\Bigg(\hat{\bm{d}}_{ij},-\frac{1}{2Q_{ij}} \bigg(\frac{\pa h_{i}}{\pa \bm{p}_{i}}\bigg) \psi_{j}-\frac{\mu}{2} \hat{\bm{d}}_{ij},l_{di}\Bigg)\nonumber\\
    &\geq& -(\bm{d}_{ij}-\hat{\bm{d}}_{ij})^{\top}\Bigg(\frac{1}{2Q_{ij}} \bigg(\frac{\pa h_{i}}{\pa \bm{p}_{i}}\bigg) \psi_{j}+\frac{\mu}{2} \hat{\bm{d}}_{ij}\Bigg).\label{projadaptive2}
\end{IEEEeqnarray}
Substituting \eqref{projadaptive1} and \eqref{projadaptive2} into \eqref{doth1} yields
\begin{IEEEeqnarray}{rCl}
\dot{\bar{h}}_{i} &\geq& \bigg(\frac{\pa h_{i}}{\pa \bm{p}_{i}}\bigg)^\top\bigg(\theta_{i}^* \bm{u}_{i}+\frac{1+\alpha_{i}}{2} \bm{u}_{i}+\sum^{N}_{j=1}\boldsymbol{d}_{ij}\psi_{j}(t)\bigg) \nonumber\\
&&-\zeta_{i}
-\tilde{\theta}_{i}\bigg(\bigg(\frac{\pa h_{i}}{\pa \bm{p}_{i}}\bigg)^{\top}\bm{u}_{i}+\mu K_{i} \hat{\theta}_{i}\bigg)\nonumber\\
&&-\sumjn \tilde{\bm{d}}_{ij}^\top\bigg(\bigg(\frac{\pa h_{i}}{\pa \bm{p}_{i}}\bigg) \psi_{j}+\mu Q_{ij}\hat{\bm{d}}_{ij}\bigg)\nonumber\\
&\geq& \bigg(\frac{\pa h_{i}}{\pa \bm{p}_{i}}\bigg)^{\!\top} \!\hat{\theta}_{i}\bm{u}_{i}
\!+\!\frac{1+\alpha_{i}}{2}\bigg(\frac{\pa h_{i}}{\pa \bm{p}_{i}}\bigg)^{\!\top} \!\bm{u}_{i}
\!+\!\bigg(\frac{\pa h_{i}}{\pa \bm{p}_{i}}\bigg)^{\!\top} \!\sumjn \hat{\bm{d}}_{ij} \psi_{j} \nonumber\\
&&-\mu K_{i} \tilde{\theta}_{i} \hat{\theta}_{i}
-\mu\sumjn Q_{ij}\tilde{\bm{d}}_{ij}^\top\hat{\bm{d}}_{ij}
-\zeta_{i}.\label{doth2}
\end{IEEEeqnarray}
Note that
\begin{IEEEeqnarray}{rCl}
\tilde{\theta}_{i} \hat{\theta}_{i} &\leq& \frac{{\theta_{i}^{*}}^2-\tilde{\theta}_{i}^2}{2}\leq \frac{{\bar\theta}_{i}^2-\tilde{\theta}_{i}^2}{2}, \label{cs1} \\
\tilde{\bm{d}}_{ij}^\top \hat{\bm{d}}_{ij} &\leq& \frac{ \bm{d}_{ij}^\top \bm{d}_{ij}-\tilde{\bm{d}}_{ij}^\top \tilde{\bm{d}}_{ij} }{2}\leq\frac{ \bar{d}_{ij}^2-\tilde{\bm{d}}_{ij}^\top \tilde{\bm{d}}_{ij} }{2}.\label{cs2}
\end{IEEEeqnarray}
The substitution of \eqref{cs1} and \eqref{cs2} into \eqref{doth2} gives
\begin{IEEEeqnarray}{rCl}
\dot{\bar{ h}}_{i} &\geq&
\bigg(\hat{\theta}_{i}+\frac{1+\alpha_{i}}{2}\bigg)\bigg(\frac{\pa h_{i}}{\pa \bm{p}_{i}}\bigg)^\top \bm{u}_{i}
+\bigg(\frac{\pa h_{i}}{\pa \bm{p}_{i}}\bigg)^{\top} \sumjn \hat{\bm{d}}_{ij} \psi_{j}
\nonumber\\
&&-\frac{\mu}{2}\bigg( K_{i} ({\bar\theta}_{i}^2 - \tilde{\theta}_{i}^2)
-\sumjn Q_{ij} (\bar{d}_{ij}^2\!-\tilde{\bm{d}}_{ij}^\top \tilde{\bm{d}}_{ij})\bigg)
-\zeta_{i} \nonumber\\
&=&
\Gamma_{i}
+ \fot \mu\bigg( K_{i}\tilde{\theta}_{i}^2+\sumjn Q_{ij} \tilde{\bm{d}}_{ij}^\top \tilde{\bm{d}}_{ij} \bigg), \label{doth3}
\end{IEEEeqnarray}
where
\begin{IEEEeqnarray}{rCl}\label{Gamma}
  \Gamma_{i}&=&\bigg(\hat{\theta}_{i}+\frac{1+\alpha_{i}}{2}\bigg)\bigg(\frac{\pa h_{i}}{\pa \bm{p}_{i}}\bigg)^\top \bm{u}_{i}+\bigg(\frac{\pa h_{i}}{\pa \bm{p}_{i}}\bigg)^\top \sumjn \hat{\bm{d}}_{ij} \psi_{j} \nonumber\\
&&-\frac{\mu}{2}\bigg( K_{i}\bar{\theta}_{i}^2+\sumjn Q_{ij} \bar{d}_{ij}^2 \bigg)
-\zeta_{i}.
\end{IEEEeqnarray}
If $\bm{u}_{i}$ in \eqref{Gamma} is selected from \eqref{kcbf}, the following condition is satisfied
\begin{equation}\label{}
  \Gamma_{i} \geq -\frac{\mu}{2}h_{i},
\end{equation}
and thus, in virtue of \eqref{hbar}, \eqref{doth3} can be reexpressed as
\begin{equation}\label{}
  \dot{\bar{h}}_{i} \geq
-\frac{\mu}{2}\bigg(
h_{i}
-K_{i}\tilde{\theta}_{i}^2
-\sumjn Q_{ij} \tilde{\bm{d}}_{ij}^\top\tilde{\bm{d}}_{ij} \bigg)
=
-\frac{\mu}{2} \bar{h}_{i}.
\end{equation}
In addition, as $\theta_{i}^{*}$ and $\bm{d}_{ij}$ are bounded by $\bar{\theta}_{i}$ and $\bar{d}_{ij}$ respectively, $\bar{h}_{i}(0)$ satisfies
\begin{IEEEeqnarray}{rCl}
\bar{h}_{i}(0) &=& h_{i}(0) -K_{i}\Big(\theta_{i}^{*}-\hat{\theta}_{i}(0)\Big)^2 \nonumber\\
&&-\sumjn Q_{ij} \Big(\bm{d}_{ij}-\hat{\bm{d}}_{ij}(0)\Big)^\top \Big(\bm{d}_{ij}-\hat{\bm{d}}_{ij}(0)\Big)\nonumber\\
&\geq& h_{i}(0)-K_{i} \Big(\bar{\theta}_{i} +\|\hat{\theta}_{i}(0)\|\Big)^2 \nonumber\\
&& -\sumjn Q_{ij} \Big(\bar{d}_{ij} +\|\hat{\bm{d}}_{ij}(0)\|\Big)^{2}.
\end{IEEEeqnarray}
The selection of parameters $K_{i}$ and $Q_{ij}$ as \eqref{adaptiveparameters} yields $\bar{h}_{i}(0)\geq 0$. According to the comparison lemma\cite{khalil2002nonlinear}, we know $\bar{h}_{i}(t)\geq0$ for $\forall t$, such that $h_{i}(t)\geq 0$ for $\forall t$.

Note that $K_{bf}$ becomes empty if $\hat{\theta}_{i}+\frac{1+\alpha_{i}}{2}=0$.
However, the design of the update law ensures $\hat{\theta}_{i}+\frac{1+\alpha_{i}}{2}>0$ for $\forall t$. The proof is as follows. As $\|\hat{\theta}_{i}(0)\|\leq \frac{1}{2}$ (see Theorem \ref{theorem1}) and $\bar{\theta}_{i}=\frac{1-\alpha_{i}}{2}$, from \eqref{l} one has $l_{\theta i}(\hat{\theta}_{i}(0))\leq 1$. According to Lemma \ref{lemma2}, we obtain $l_{\theta i}(\hat\theta)\leq 1$, which indicates $\|\hat{\theta}_{i}\|\leq \bar{\theta}_{i}+\frac{\alpha_{i}}{2}=\frac{1-\alpha_{i}}{2}+\frac{\alpha_{i}}{2}=\fot$ for $\forall t$. Hence, $\hat{\theta}_{i}\in[-\fot,\fot]$, such that $\hat{\theta}_{i}+\frac{1+\alpha_{i}}{2} \in [\frac{\alpha_{i}}{2},1+\frac{\alpha_{i}}{2}]$ for $\forall t$.
\end{proof}

By Theorem \ref{theorem1}, a safe controller is obtained by solving the following CBF-QP problem
\begin{align}
\min_{\bm{u}_{i}^{*}} \quad & \|\bm{u}_{i}-\hat{\bm{u}}_{i}\|^2\label{cbfQP}\tag{CBF-QP}\\
\textrm{s.t.}\quad &  \bigg(\frac{\pa h_{i}}{\pa \bm{p}_{i}}\bigg)^\top \sum_{i=1}^N \hat{\bm{d}}_{ij} \psi_{j}
+\bigg(\hat{\theta}_{i}+\frac{1+\alpha_{i}}{2}\bigg)\bigg(\frac{\pa h_{i}}{\pa \bm{p}_{i}}\bigg)^\top \bm{u}_{i} \nonumber\\
&-\zeta_{i}+\frac{\mu}{2}\bigg(h_{i}-K_{i}\bar{\theta}_{i}^2-\sumjn Q_{ij} \bar{d}_{ij}^2\bigg)\geq 0,
\end{align}
where $\hat{\bm{u}}_{i}$ represents a nominal control law that may be unsafe in a coverage problem with non-zero sized agents.

\section{Case study}\label{sec:case}
For the verification of the proposed CBF based controller, a safety critical control scenario is considered in the overage problem. Specifically, each agent has a safety range which should not overlap with that of other agents.
The following density function is adopted for the region to be covered.
Given $\boldsymbol{q}\in \mathcal{\boldsymbol{D}}$ where
$\boldsymbol{q}=
  \begin{bmatrix}
    x & y
  \end{bmatrix}^\top$,
a time-invariant density function $\phi(\boldsymbol{q})$ is defined by a bivariate normal distribution
\begin{equation}\label{}
  \phi=\frac{1}{2\pi\sigma_{x}\sigma_{y}}
  \exp\Bigg(-\frac{1}{2}\bigg(\frac{(x-\mu_{x})^{2}}{\sigma_{x}^{2}}+\frac{(y-\mu_{y})^{2}}{\sigma_{y}^{2}}\bigg)\Bigg),
\end{equation}
where $\sigma_{x}=\sigma_{y}=0.3$, $\mu_{x}=\mu_{y}=1.75$. It can be visualized by Fig.~\ref{pic:density}.
The agents are expected to move towards the portion with a higher density while covering the region $\mathcal{\boldsymbol{D}}$.

To demonstrate the effectiveness of collision avoidance for multi-agent systems in the presence of uncertainties, a comparison between results under the nominal controller (see Fig.~\ref{snapshots_cvt}, Fig.~\ref{dismin_cvt}, and Fig.~\ref{path_cvt}) and proposed CBF based controller (see Fig.~\ref{snapshots_cbf}, Fig.~\ref{dismin_cbf}, and Fig.~\ref{path_cbf}) is conducted. The nominal controller is selected as the non-switched controller designed in\cite{Bai_22a}.
The blue polygons in Fig.~\ref{snapshots_cvt} and Fig.~\ref{snapshots_cbf} represent the region $\mathcal{\boldsymbol{D}}$ to be covered. The centers of agents are denoted by the black dots, around which there are safety ranges with a radius $r_{\textrm{safe}}=0.25m$ around itself, represented by red circles.
The time histories for the minimum distance between all pairs of agents' centers and paths of agents are respectively plotted in Fig.~\ref{dismin_cvt}, Fig.~\ref{path_cvt} for nominal control, and in Fig.~\ref{dismin_cbf}, Fig.~\ref{path_cbf} for CBF control.

Both the actuator faults and the time-varying disturbances are taken into account to test the adaptiveness of the CBF based controller. The actuator fault coefficient $\theta_{i}=\frac{1}{2}$, and the disturbance $\boldsymbol{d}_{i}$ to the control system is chosen as
$
\boldsymbol{d}_{i}(t)=
  \begin{bmatrix}
    d_{x}(t) & d_{y}(t)
  \end{bmatrix}^\top
$,
where
\begin{equation}\label{}
  d_{x}(t) = d_{y}(t)=
  \begin{cases}
		\frac{{d}_{\textrm{max}}}{2}t,\indent  {0\leq{t}< \frac{T}{6}}, \\
		{{d}_{\textrm{max}}{t}},\indent  {\frac{T}{6}\leq{t}<\frac{T}{3}}, \\
		\frac{{d}_{\textrm{max}}}{2}(\frac{T}{2} - t),\indent  {\frac{T}{3}\leq{t}<\frac{2T}{3}}, \\
		{-{d}_{\textrm{max}}},\indent  {\frac{2T}{3}\leq{t}<\frac{5T}{6}}, \\
		\frac{{d}_{\textrm{max}}}{2}{(t - T)},\indent  {\frac{5T}{6}\leq{t}\leq T},
	\end{cases}
\end{equation}
and the maximum amplitude of disturbance $d_{\textrm{max}}=1$.
Other parameters in the simulation are selected as
$n=8$, $T=30s$, $l=5$, $\bar{d}_{ij}=20$, $E=0.2$, $V_{z}=10$, $\alpha_{i}=\nu_{i}=0.1$, $\mu=2$.

Fig.~\ref{snapshots_cvt} and Fig.~\ref{dismin_cvt} show that under the nominal control, the safety ranges of agents are overlapped with each other, which indicates possible collisions between agents. On the hand, Fig.~\ref{snapshots_cbf} and Fig.~\ref{dismin_cbf} demonstrate that the minimum distance between all pairs of agents is always larger than or equal to the safety distance ($2 r_{\textrm{safe}}$), even in the presence of actuator faults and time-varying disturbances.
\begin{figure}[!ht]
  \centering
  \includegraphics[width=6.5cm]{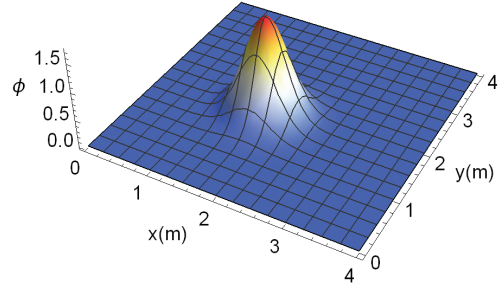}\\
  \caption{Density function.}\label{pic:density}
\end{figure}
\begin{figure*}[ht]
    \centering
    \subfigure[\label{ini_constant}]{
    \includegraphics[height=4cm]{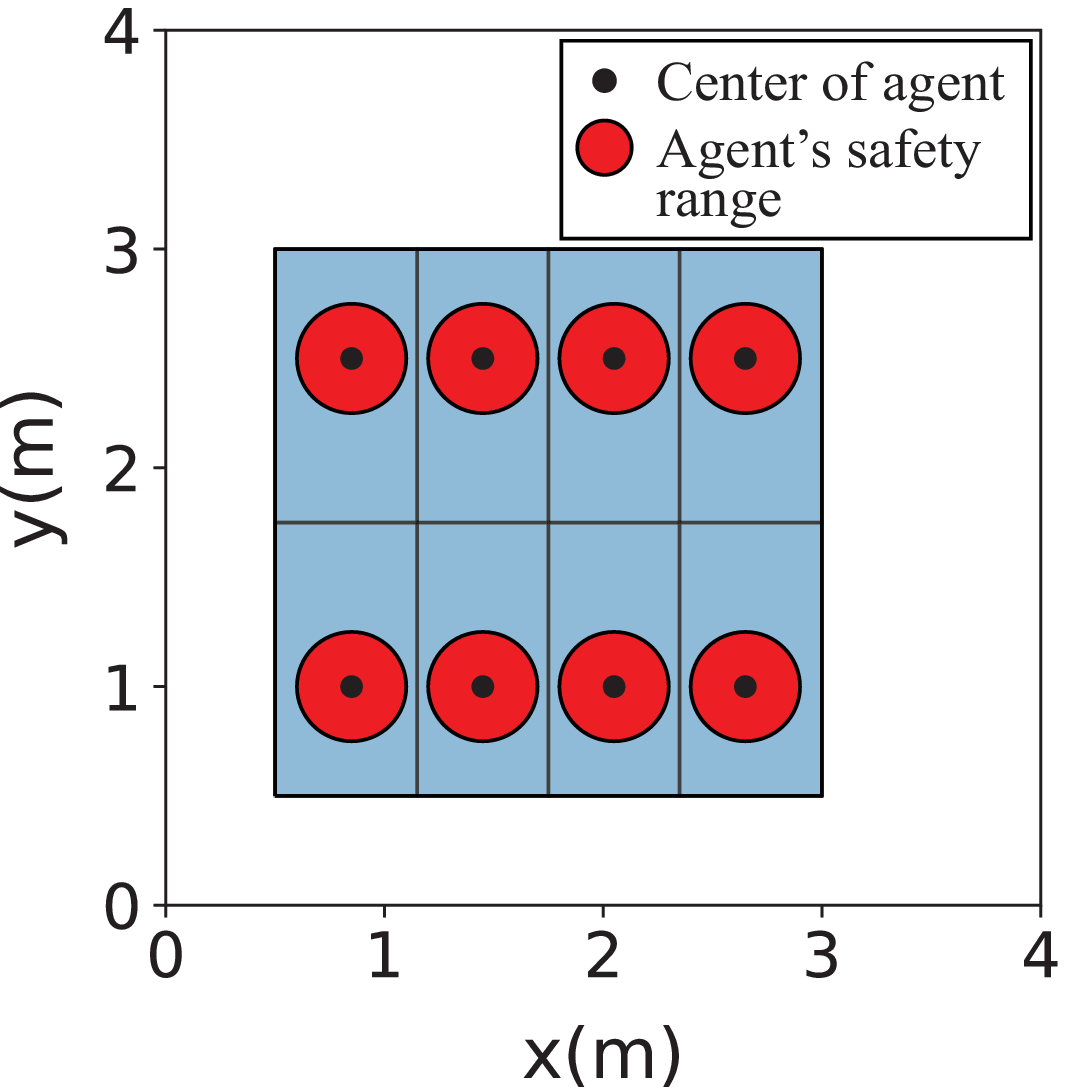}
}\hspace{0.2cm}
    \subfigure[]{
    \includegraphics[height=4cm]{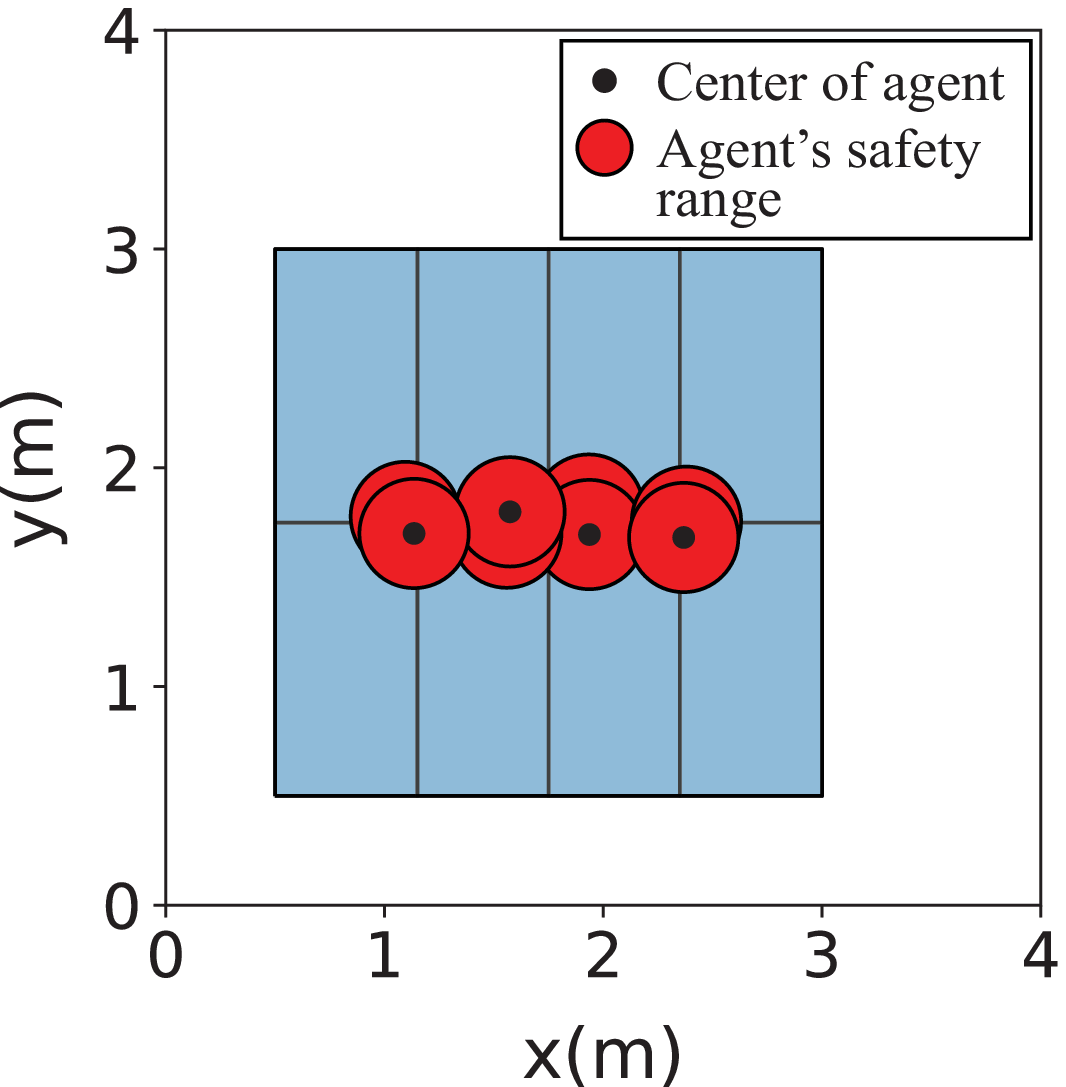}
}\hspace{0.2cm}
    \subfigure[]{
	\includegraphics[height=4cm]{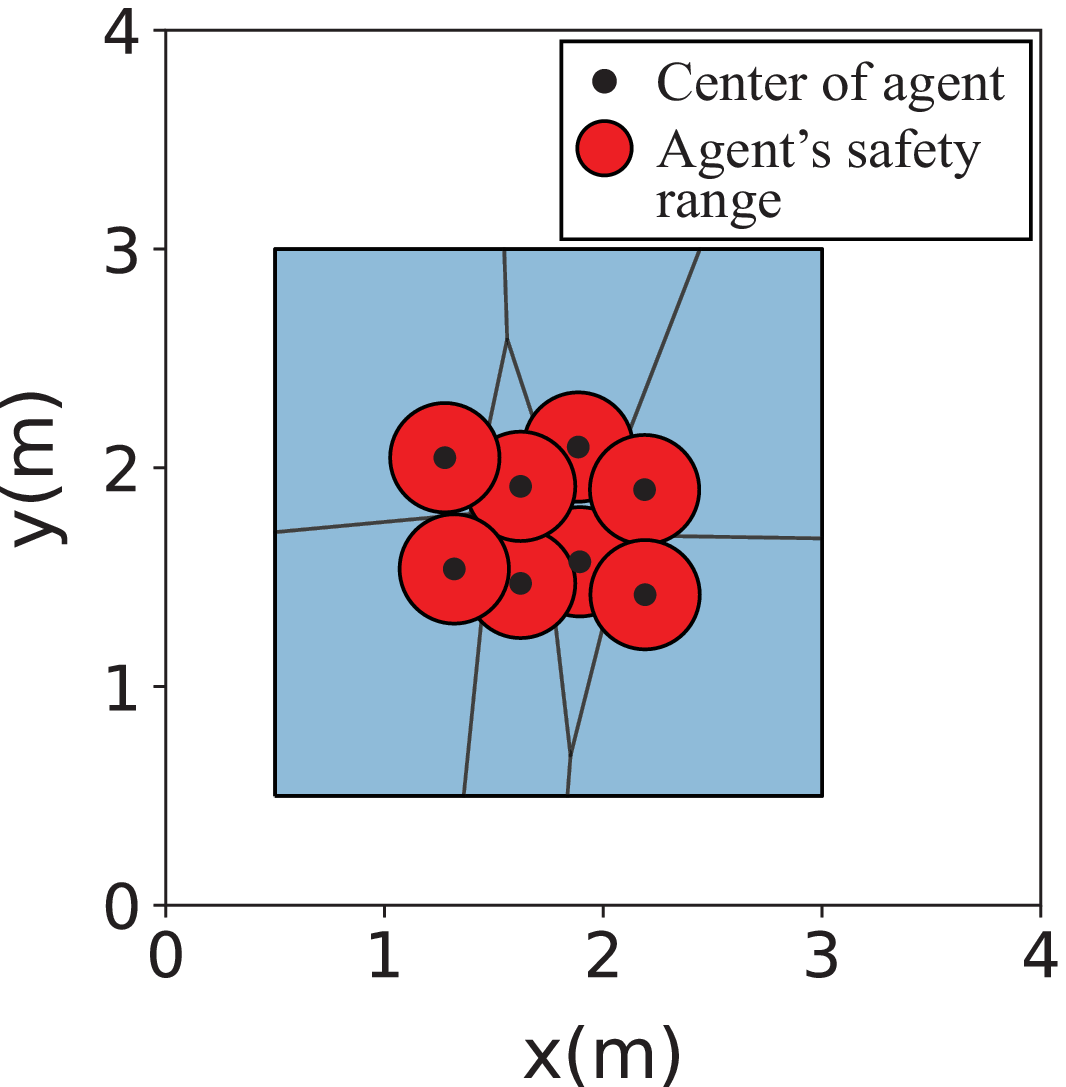}
}\hspace{0.2cm}
    \subfigure[]{
    \includegraphics[height=4cm]{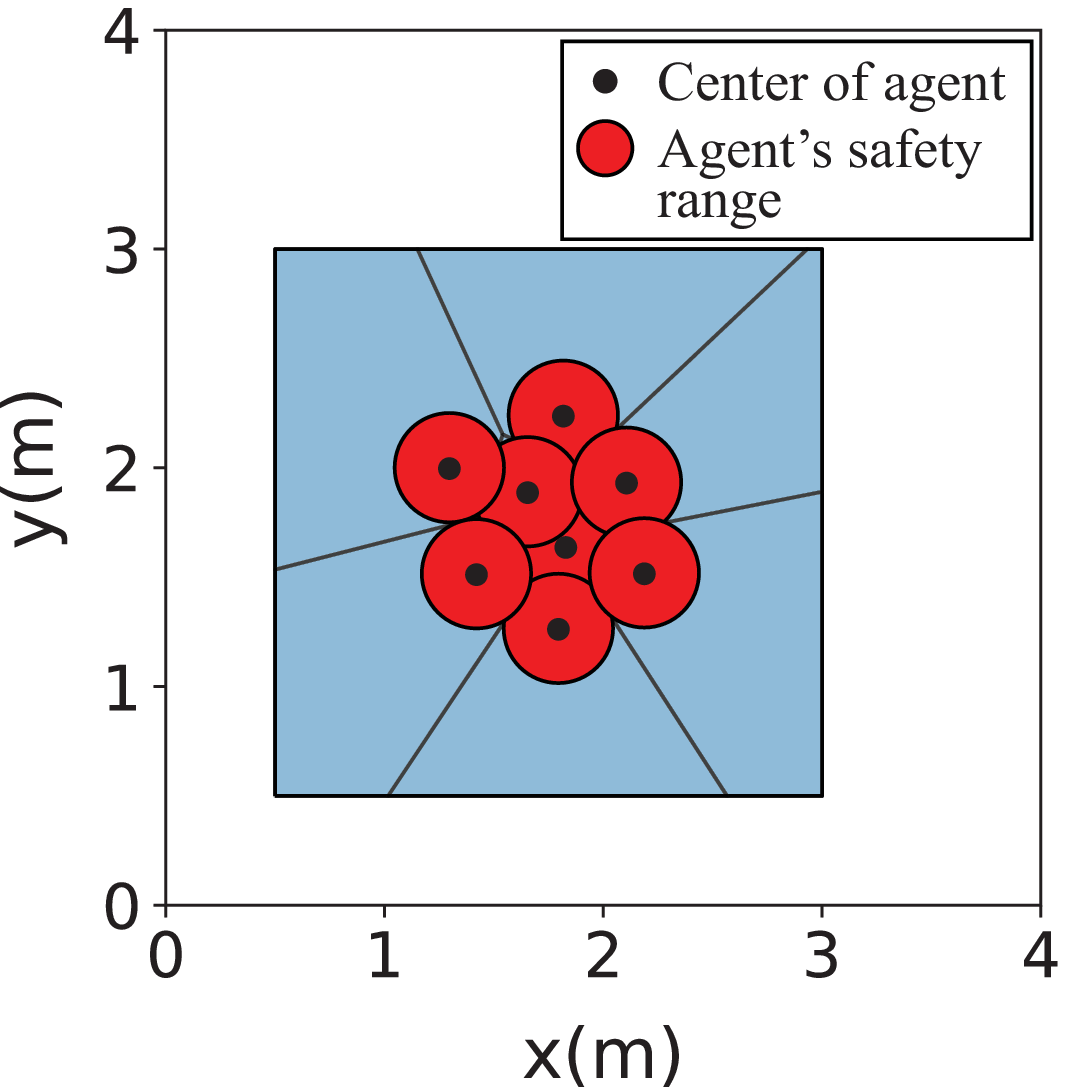}
}
    \caption{Snapshots for the multi-agent coverage under nominal control (without CBF): (a) t=0s, (b) t=1s, (c) t=10s, (d) t=30s.
    }
\label{snapshots_cvt}
\end{figure*}

\begin{figure*}[ht]
    \centering
    \subfigure[\label{ini_vary}]{
    \includegraphics[height=4cm]{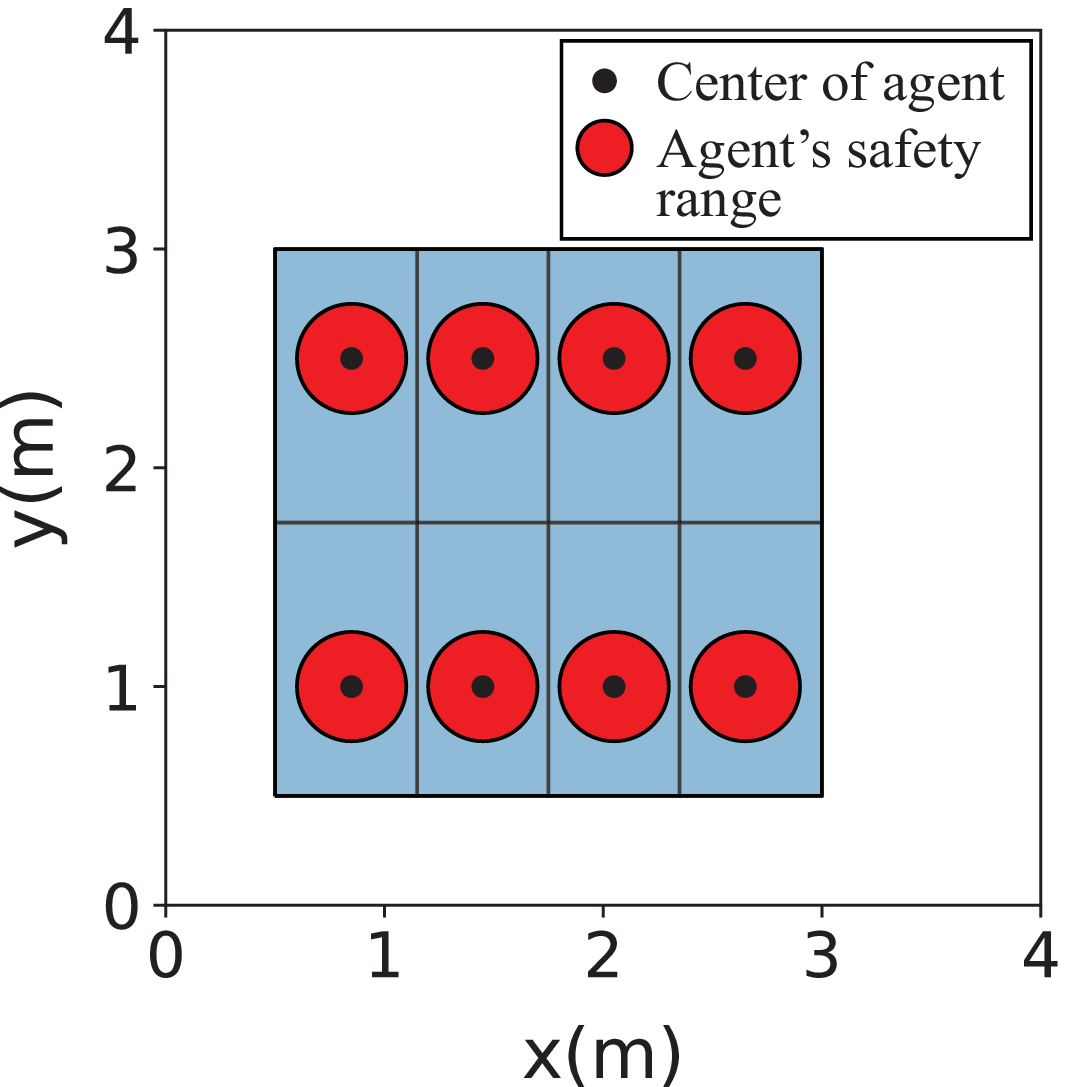}
}\hspace{0.2cm}
    \subfigure[]{
    \includegraphics[height=4cm]{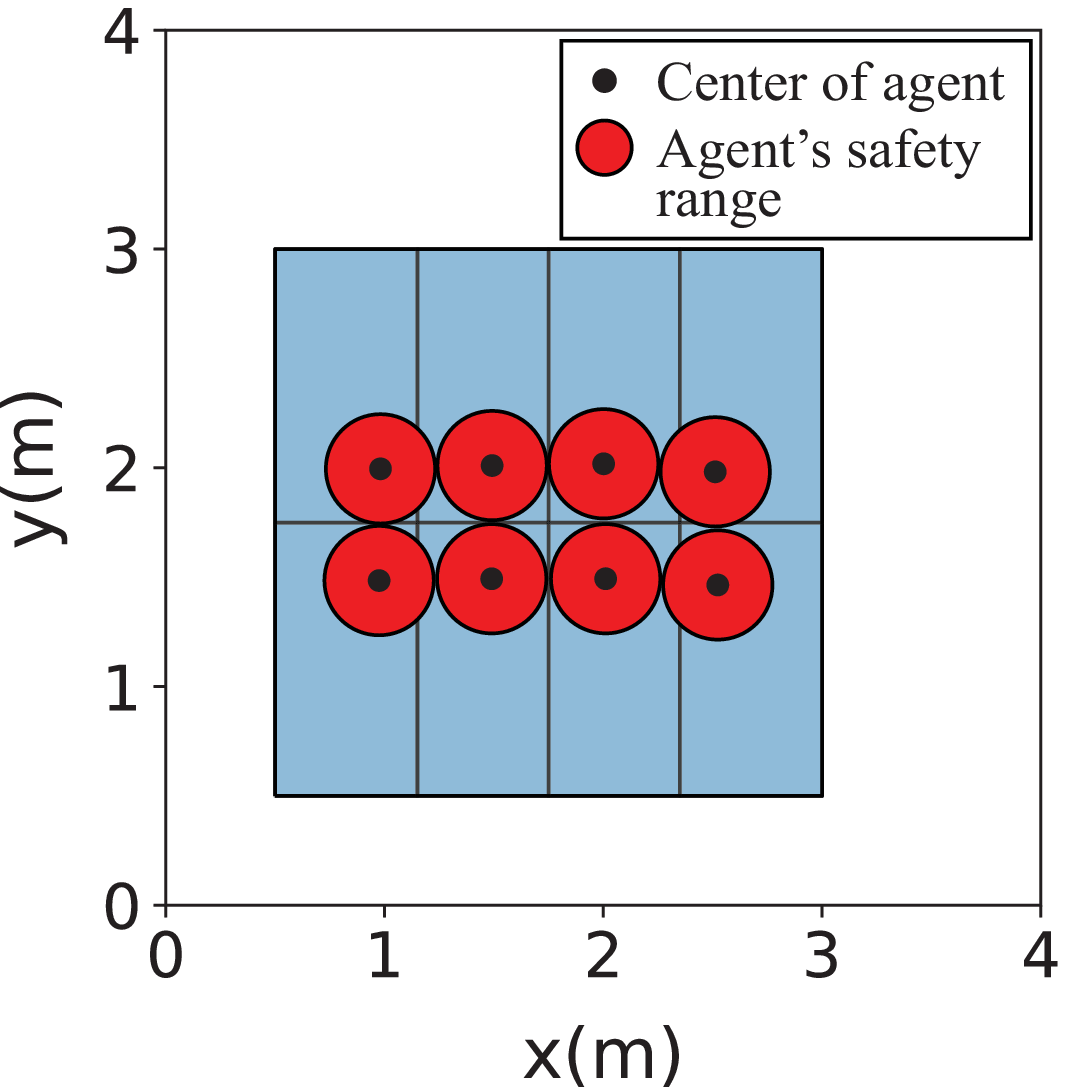}
}\hspace{0.2cm}
    \subfigure[]{
	\includegraphics[height=4cm]{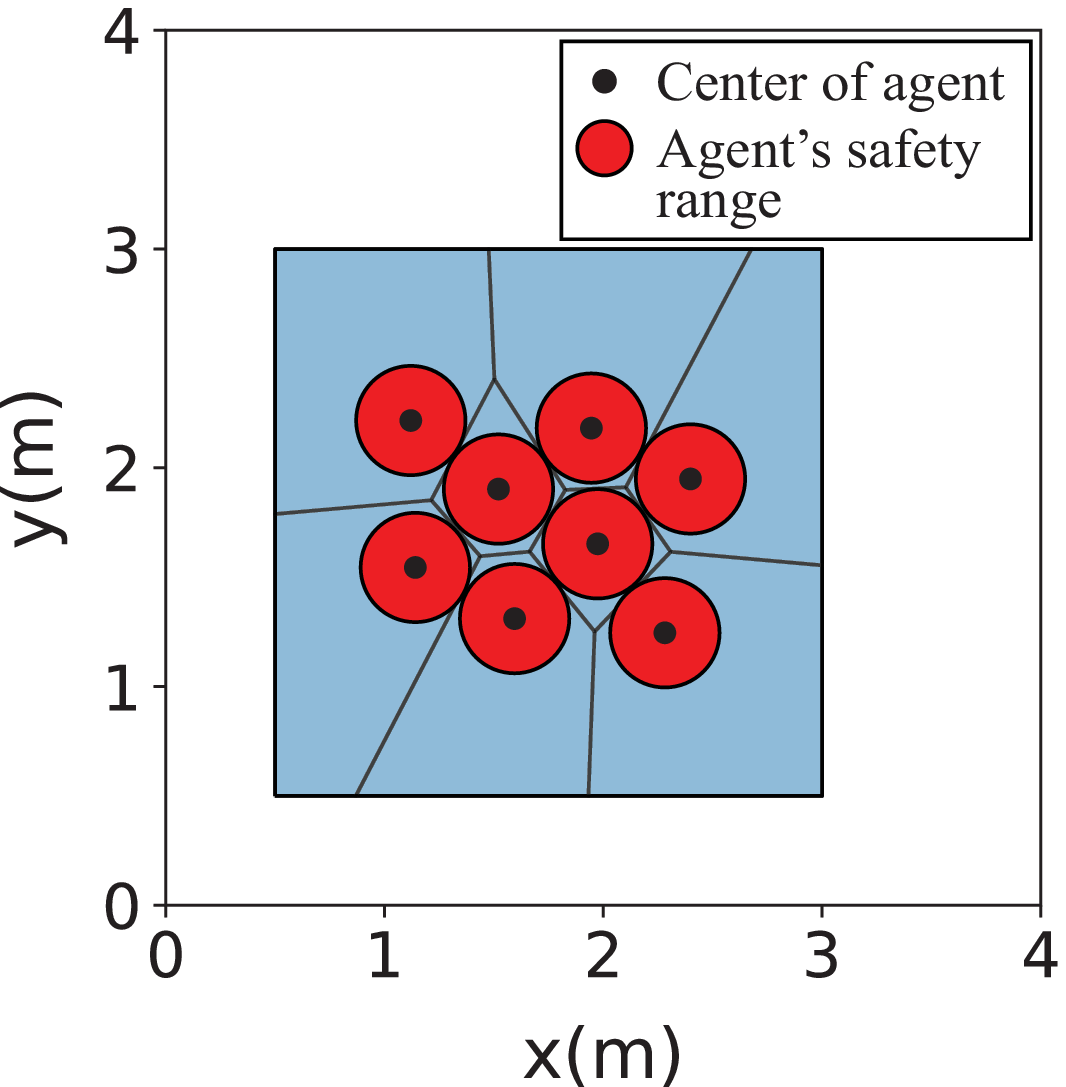}
}\hspace{0.2cm}
    \subfigure[]{
    \includegraphics[height=4cm]{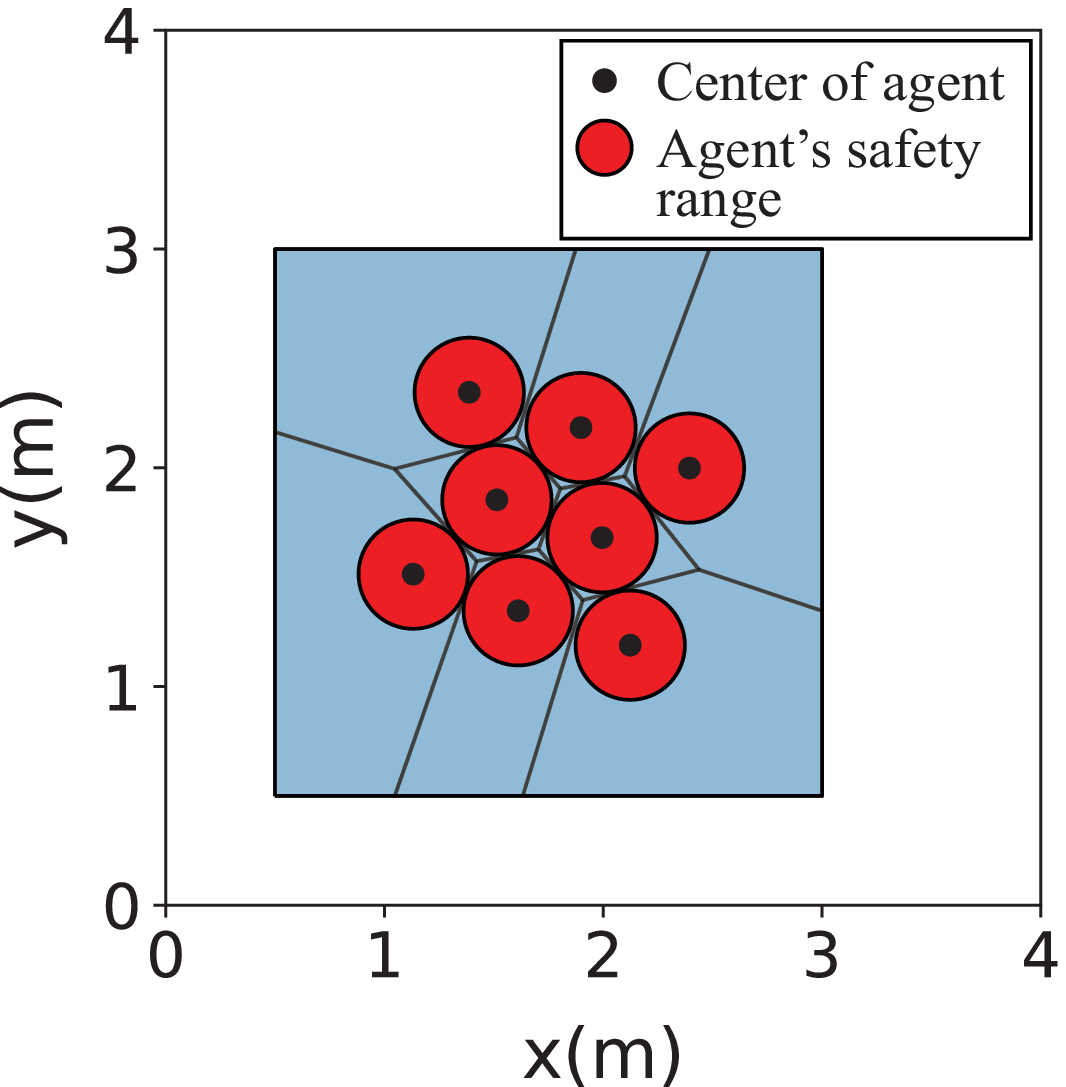}
}
    \caption{Snapshots for the multi-agent coverage under the CBF based control: (a) t=0s, (b) t=1s, (c) t=10s, (d) t=30s.
    }
\label{snapshots_cbf}
\end{figure*}

\begin{figure*}[!ht]
    \centering
    \subfigure[\label{dismin_cvt}]{
    \includegraphics[height=3.6cm]{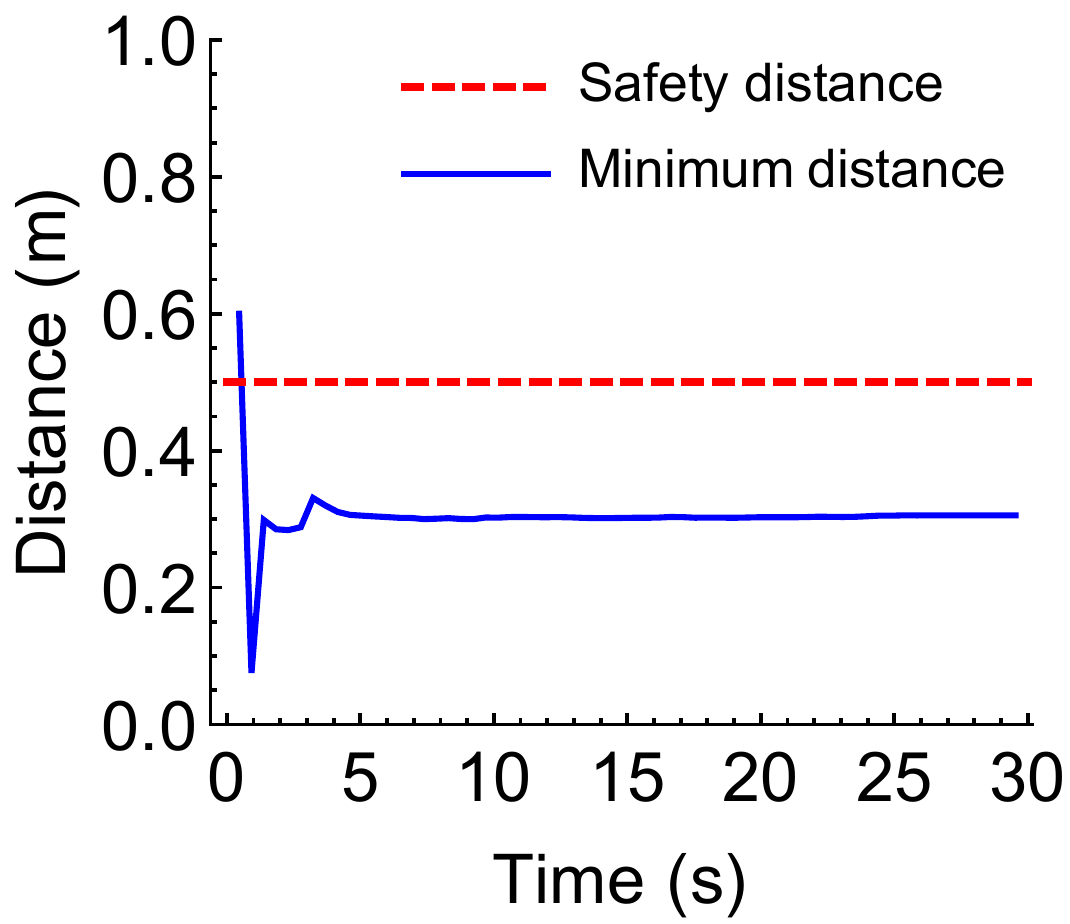}
}
    \subfigure[\label{path_cvt}]{
    \includegraphics[height=3.6cm]{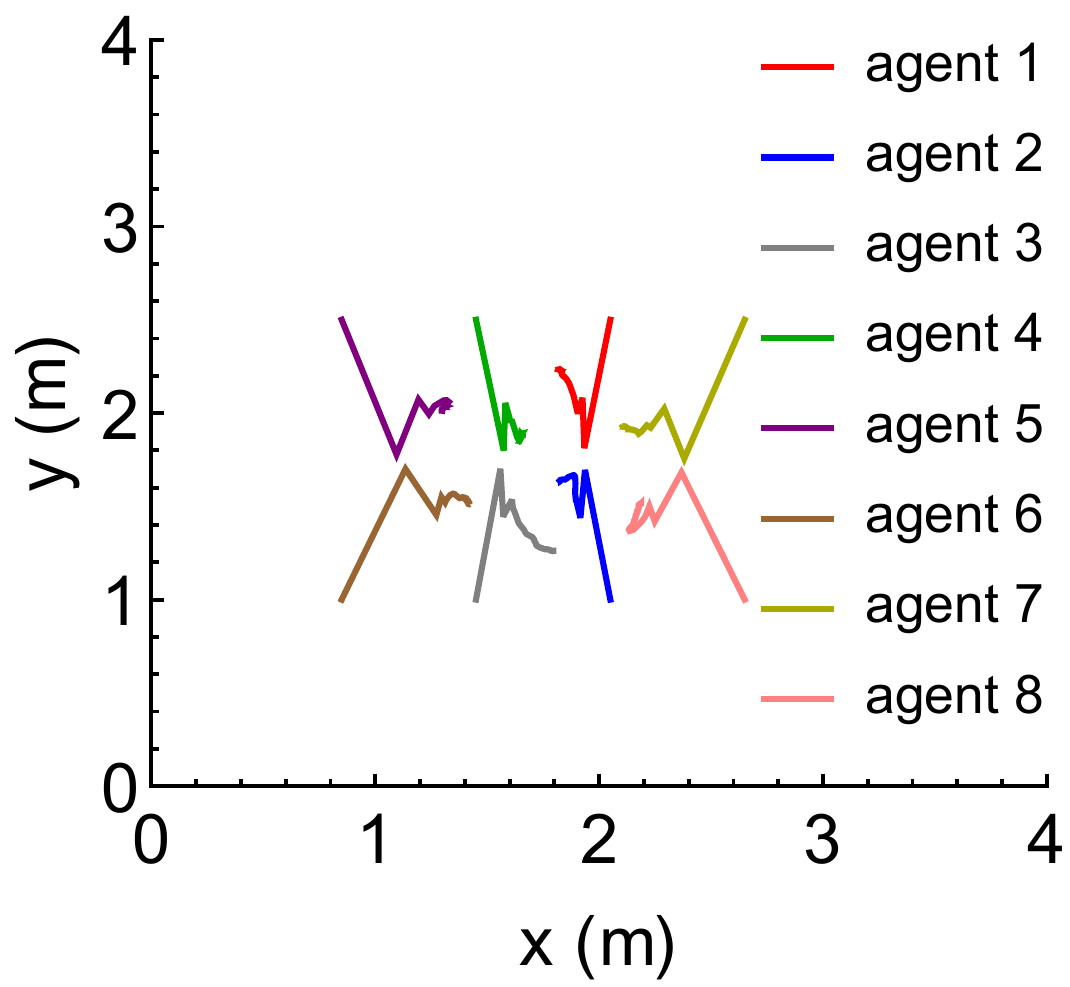}
}
    \subfigure[\label{dismin_cbf}]{
    \includegraphics[height=3.6cm]{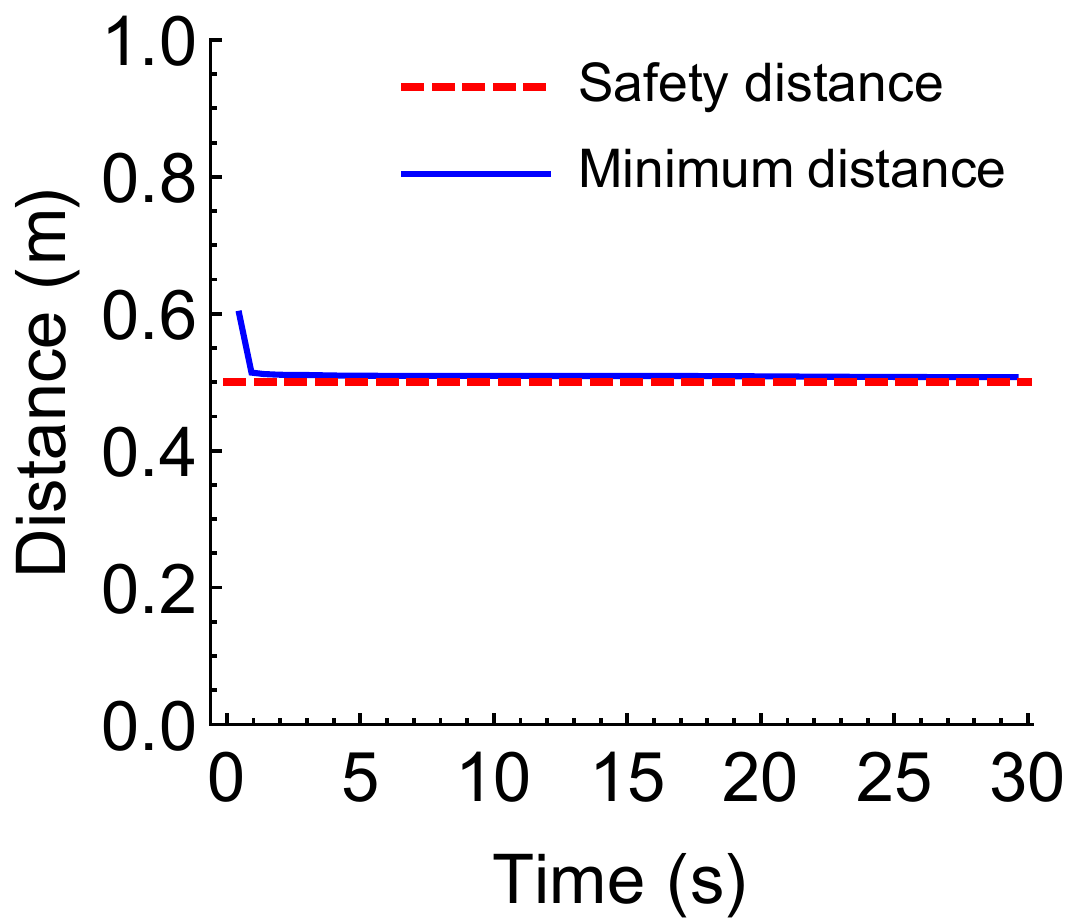}
}
    \subfigure[\label{path_cbf}]{
    \includegraphics[height=3.6cm]{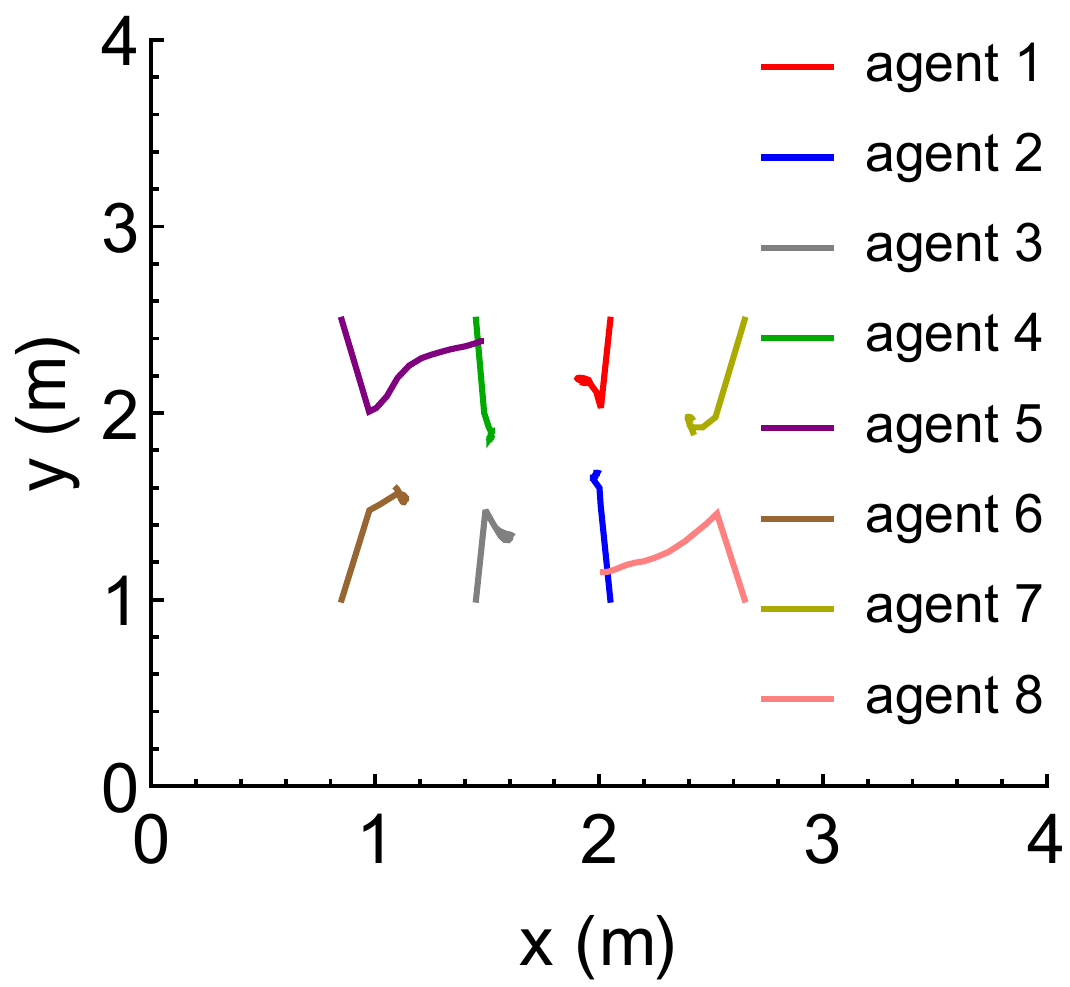}
}
    \caption{(a) Minimum distance between all pairs of agents' centers under nominal control and (b) corresponding paths. (c) Minimum distance under the CBF based control and (d) corresponding paths.}
\label{fig3}
\end{figure*}

\section{Conclusions}\label{sec:concl}
A safe adaptive coverage problem has been studied in the presence of actuator faults and time-varying uncertainties. An optimal coverage configuration for a multi-agent system has been generated though CVT. The agents have been adaptively driven to the generated configuration by the proposed FAT based controller, and collision avoidance between agents has been guaranteed by a CBF based design. The proposed controller has been verified under simulations and we also plan to test its validity by experiments in the future work.
%



\end{document}